\documentclass[11pt,a4paper]{amsart}
\usepackage{amsopn}
\usepackage{amsmath,amsthm,amssymb,color,enumerate}
\usepackage{prettyref}

\usepackage{pdfsync}
\newcommand{\nc}{\newcommand}

 \nc{\aff}{\mathfrak{aff} } \nc{\bb}{\mathfrak{b}}
\nc{\cc}{\mathfrak{c} }  \nc{\dd}{\mathfrak{d} }
 \nc{\ggo}{\mathfrak{g} }
 \nc{\hh}{\mathfrak{h} }  \nc{\ii}{\mathfrak{i} }
 \nc{\jj}{\mathfrak{j} }  \nc{\kk}{\mathfrak{k} }
\nc{\mm}{\mathfrak{m} }   \nc{\nn}{\mathfrak{n} }
\nc{\pp}{\mathfrak{p} }  \nc{\rr}{\mathfrak{r} } \nc{\sg}{\mathfrak{s} }
 \nc{\sog}{\mathfrak{so} }  \nc{\spg}{\mathfrak{sp}}
 \nc{\sug}{\mathfrak{su} }  \nc{\slg}{\mathfrak{sl}}
 \nc{\tg}{\mathfrak{t} }  \nc{\uu}{\mathfrak{u} }
 \nc{\vv}{\mathfrak{v} } \nc{\ww}{\mathfrak{w} }
 \nc{\zz}{\mathfrak{z} }

 \nc{\ggob}{\overline{\mathfrak{g}}}

\nc{\glg}{\mathfrak{gl} }

\nc{\pca}{\mathcal{P}} \nc{\nca}{\mathcal{N}}

 \nc{\vp}{\varphi} \nc{\ddt}{\frac{{\rm d}}{{\rm d}t}}
 \nc{\la}{\langle} \nc{\ra}{\rangle}

 \nc{\SO}{{\sf SO}} \nc{\Spe}{{\sf Sp}} \nc{\Sl}{{\sf Sl}}
 \nc{\SU}{{\sf SU}} \nc{\Or}{{\sf O}} \nc{\U}{{\sf U}}
 \nc{\Gl}{{\sf Gl}} \nc{\Se}{{\sf S}} \nc{\Cl}{{\sf Cl}}
 \nc{\Spin}{{\sf Spin}} \nc{\Pin}{{\sf Pin}}

 \nc{\RR}{{\mathbb R}} \nc{\HH}{{\mathbb H}} \nc{\CC}{{\mathbb C}}
 \nc{\ZZ}{{\mathbb Z}} \nc{\FF}{{\mathbb F}} \nc{\NN}{{\mathbb N}}
 \nc{\GG}{{\mathbb G}} \nc{\JJ}{{\mathbb J}} \nc{\II}{{\mathbb I}}
 \nc{\KK}{{\mathbb K}} \nc{\DD}{{\mathbb D}}

 \nc{\ad}{\operatorname{ad}} \nc{\Ad}{\operatorname{Ad}}
 \nc{\coad}{\operatorname{coad}} \nc{\ct}{\operatorname{T}}
 \nc{\rank}{\operatorname{rank}} \nc{\Irr}{\operatorname{Irr}}
 \nc{\End}{\operatorname{End}} \nc{\Aut}{\operatorname{Aut}}
 \nc{\Inn}{\operatorname{Inn}} \nc{\Der}{\operatorname{Der}}
 \nc{\Dera}{\operatorname{Dera}} \nc{\Auto}{\operatorname{Auto}}
 \nc{\GL}{\operatorname{GL}}
 \nc{\SL}{\operatorname{SL}}
 \nc{\coord}{\operatorname{coord}}

 \theoremstyle{plain}
 \newtheorem{teo}{Theorem}[section]
 \newtheorem{pro}[teo]{Proposition}
 \newtheorem{cor}[teo]{Corollary}
 \newtheorem{lm}[teo]{Lemma}

 \renewenvironment{proof}{\noindent \emph{Proof.}}{\hfill $\blacksquare$}

 \theoremstyle{remark}
 
 \newtheorem{rem}[teo]{Remark}
 \newtheorem{defi}[teo]{Definition}

 \newcommand{\R}{\mathbb R}

\newcommand{\Z}{\mathbb Z}

\newcommand{\mg}{\mathfrak g }
\newcommand{\mb}{\mathfrak b }

\newcommand{\mn}{\mathfrak n }

\newcommand{\mh}{\mathfrak h }

\newcommand{\mgg}{\mathfrak g }
\newcommand{\mpp}{\mathfrak p }

\newcommand{\tmn}{ \tilde{\mn} }
\newcommand{\tmm}{ \tilde{\mm} }

\nc{\rmc}{\textrm{\rm C}}

\newcommand{\lra}{\longrightarrow}

\begin{document}
\title[Nilradicals admitting symplectic structures]{Nilradicals of parabolic subalgebras admitting symplectic structures}

\author{Leandro Cagliero}
\address[L. Cagliero]{CONICET - Universidad Nacional de C\'ordoba, FAMAF, Av. Medina Allende s/n, 5000 C\'ordoba, C\'ordoba,  Argentina.}
\email{cagliero@famaf.unc.edu.ar}

\author{Viviana del Barco}
\address[V. del Barco]{CONICET - Universidad Nacional de Rosario, ECEN-FCEIA, Depto. de Matem\'a\-tica, Av. Pellegrini 250, 2000 Rosario, Santa Fe, Argentina.}
\email{delbarc@fceia.unr.edu.ar}

 \date{\today}

\begin{abstract} 
In this paper we describe all the nilradicals of 
parabolic subalgebras of split real simple Lie algebras admitting symplectic structures.

The main tools used to obtain this list are Kostant's description of the highest weight vectors
(hwv) of the cohomology of these nilradicals and some necessary conditions obtained
for the $\mgg$-hwv's of $H^2(\mn)$ for a finite dimensional real symplectic nilpotent Lie algebra $\mn$ 
with a reductive Lie subalgebra of derivations $\mgg$ acting on it. 
\end{abstract}

\thanks{This work was partially supported by SECyT-UNR, SECyT-UNC, CONICET and FONCyT. \\
{\it (2010) Mathematics Subject Classification}:  17B20, 17B30, 17B56, 53D05.\\
{\it Keywords: Nilpotent Lie algebras, symplectic structures, nilradicals of parabolic subalgebras.}
}


\maketitle

\section{Introduction}
Throughout this paper, Lie algebras are always finite dimensional over $\R$.

The problem of determining whether an arbitrary 
Lie algebra admits symplectic structures is difficult in general and 
only few convenient necessary and sufficient conditions are known. 
The problem is well understood only for some specific families of Lie algebras, 
for instance some subfamilies of filiform Lie algebras \cite{GJK,MI}, Heisenberg Lie algebras \cite{DT}, 
nilpotent Lie algebras associated with graphs \cite{PO-TI}, 
free nilpotent Lie algebras \cite{dB1}, among others.

Let  $\mn=\mn^0\supset\mn^1\supset\dots\supset\mn^{k-1}\supset\mn^{k}=0$ 
be a filtration of a nonabelian nilpotent Lie algebra, invariant under the 
action of a reductive subalgebra $\mgg$ of $\Der(\mn)$.
We say that this filtration is {\em accurate} if 
it verifies some compatibility condition with respect to closed and exact 2-forms 
of $\mn$ (see Definition \ref{def.accurate}). 
Any such filtration gives rise to $\mg$-submodules $\mm_j\subset \mn^*$ ($j=1,\dots,k$) 
that are compatible with the given filtration, and a family of projections 
$P_t:\Lambda^2 \mn^*\to \mm_t\wedge\mm_{k-t+1}$ which are associated 
to the decomposition $\Lambda^2\mn^* = \bigoplus_{ 1\le i\le j\le k} \mm_i\wedge \mm_j$.  
Our first result is Theorem \ref{thm.Criterio2} which shows that when $\mn$ or $\R\oplus\mn$ 
is symplectic and, for some $t$,  $\dim \mn^{k-t} + \dim \mn^{t-1}>\dim \mn$
(it is always true for $t=1$), then there is an irreducible 
$\mgg$-submodule $U$ of $H^2(\mn)$ such that $P_{t}(\sigma)\ne0$ for any $\sigma\in\Lambda^2\mn^*$ 
corresponding to a nonzero cohomology class in $U$.

This result, and Kostant's description of the module structure 
of the cohomology,
allow us to determine the set of all nilradicals of 
parabolic subalgebras of split real simple Lie algebras admitting symplectic structures as follows. 
As a first step we obtain in Theorem \ref{prop.H2_hwv} 
a very explicit description of the set of highest weight vectors
of $H^2(\mn)$ (here $\mn$ is the nilradical of the parabolic subalgebra $\mgg\ltimes\mn$) 
given by Kostant in his description of its $\mgg$-module structure.
Finally,  the list of symplectic nilradicals is obtained 
by applying Theorem \ref{thm.Criterio2} to the lower central series of $\mn$
and the Levi factor $\mgg$.

We prove in Theorem \ref{thm.class} that if $\mn$ is one of these nilradicals and either 
$\mn$ or $\R\oplus\mn$ admits a symplectic structure, then $\mn$ belongs to the following list.

\begin{enumerate}[1.]
 \item The abelian Lie algebras.
 \item The free 2-step nilpotent Lie algebra generated by 2 or 3 elements.
 \item The filiform Lie algebra of dimension 4: $[X_1,X_2]=X_3$, $[X_1,X_3]=X_4$.
 \item $\{X_i,Y_i,\,Z_{ij}:1\leq i ,  j\leq 2\}$, $[X_i,Y_j]=Z_{ij}$, $1\!\leq\! i ,  j\!\leq\! 2$.
 \item $\{X_i,Y_i,\,Z_{ij}:1\!\leq\! i\!\le\! j\!\leq\! n\}$, $[X_i,Y_j]=[X_j,Y_i]=Z_{ij}$, $1\!\leq\! i\!\le\! j\!\leq\! n$; $n=2,3$.
 \item $\{X_i,Y_i,\,Z_{ij}:1\!\leq\! i\! <\!  j\!\leq\! 3\}$, $[X_i,Y_j]=-[X_j,Y_i]=Z_{ij}$, $1\!\leq\! i\! <\!  j\!\leq\! 3$.
 \item $\R X\ltimes_{\ad_X}(\R^{n-1}\oplus\R^{n-1})$ 
with $
\ad_X=\left(\begin{smallmatrix}
0&0\\[1mm]
I&0
\end{smallmatrix}\right)
$.
\end{enumerate}
All these Lie algebras are 2-step nilpotent, except for the 3-step filiform Lie algebra.

The results here extend those in \cite{dB3} where the author  
obtains the list of symplectic nilradicals of 
Borel subalgebras of split real simple Lie algebras. The main tool there is a general result 
about symplectic nilpotent Lie algebras involving the spectral sequence arising from the filtration 
associated to its 
lower central series. 
This result is contained here, with different notation, as the particular case $t=1$ and $\mgg=0$ of Corollary \ref{cor.seriecent}. 

Taking into account how difficult is to prove the non-existence 
of symplectic structures, we think that Theorem \ref{thm.Criterio2} gives an interesting 
(and hopefully powerful) tool to attack this problem.

\section{Symplectic structures on  Lie algebras}
In this section we briefly recall some basic facts about Lie algebra cohomology and present two useful lemmas about symplectic Lie algebras

Let $\mgg$ be a real Lie algebra of dimension $m$. The Chevalley-Eilenberg complex of $\mgg$ is the differential complex
\begin{equation}\label{eq:complex}
0\lra\R \lra \mgg^*\stackrel{d_1}{\lra}\Lambda^2\mgg^*\stackrel{d_2}{\lra}\ldots \ldots\stackrel{d_{m-1}}{\lra} \Lambda^m\mg^*\lra 0,
\end{equation}
where we identify the exterior product $\Lambda^p\mgg^*$ with the space of skew-symmetric $p$-linear forms on $\mgg$. 
Thus each differential $d_p:\Lambda^{p}\mgg^*\lra \Lambda^{p+1}\mgg^*$ is defined by:
\begin{equation} d_p c\,(x_1,\ldots,x_{p+1})=\sum_{1\leq i<j\leq p+1}(-1)^{i+j-1}c([x_i,x_j],x_1,\ldots,\hat{x_i},\ldots,\hat{x_j},\ldots,x_{p+1}). \label{eq:diff}\end{equation}

The first differential $d_1$ coincides with the dual 
mapping of the Lie bracket $[\,,\,]:\Lambda^2\mgg\lra \mgg$ and 
$d=\bigoplus_p d_p$ is a derivation of the exterior algebra $\Lambda^*\mgg^*$ such that $d^2=0$. 
The cohomology of $(\Lambda^*\mgg^*,d)$ is called the Lie algebra cohomology of $\mgg$ (with real coefficients) and it is denoted by $H^*(\mgg)$.

A given $\omega\in\Lambda^*\mgg^*$ is said to be closed if $d\omega=0$ and 
it is said to be exact if $\omega=d\nu$ for some $\nu\in\Lambda^*\mgg^*$.
In particular, if $\omega$ is a 2-form in $\mgg$, that is $\omega\in\Lambda^2\mgg^*$, then 
$\omega$ is closed if and only if (see \eqref{eq:diff}) 
\begin{equation}\label{eq:closed}
\omega([U,V],W)+\omega([V,W],U)+\omega([W,U],V)=0,\quad \text{ for all }U,V,W\in\mgg,
\end{equation}
and $\omega$ is exact if and only if there is some $f\in\mgg^*$ such that 
$\omega(U,V)=-f([U,V])$ for all $U,V\in\mgg$.

A symplectic structure on $\mgg$ is a closed 2-form $\omega$ which is also nondegenerate. 
It is clear that symplectic Lie algebras are necessarily even dimensional. 
If $\dim\mgg=m$, then a closed 2-form $\omega$ is a symplectic structure on $\mgg$ 
if and only if $\omega^{m/2}\neq0$.

The lower central series of a Lie algebra $\mgg$ is $\{\rmc^j(\mgg)\}$, where for all $j\geq 0$, $\rmc^j(\mgg)$ is defined by
$$\rmc^0(\mgg)=\mg,\quad  \rmc^j(\mgg)=[\mg,\rmc^{j-1}(\mgg)],\;\; j\geq 1.
$$
The upper central series of $\mgg$ is $\{\rmc_j(\mgg)\}$, where for all $j\geq 0$, $\rmc_j(\mgg)$ is given by
$$\rmc_0(\mgg)=0,\quad  \rmc_j(\mgg)=\{X\in\mg:\,[X,\mg]\subset \rmc_{j-1}(\mgg)\},\;\; j\geq 1.
$$
We notice that $\rmc^1(\mgg)=[\mg,\mg]$ is the commutator of $\mg$ and $\rmc_1(\mgg)$ is the center of $\mg$.

The following lemma gives an obstruction for a Lie algebra to admit symplectic structures (see \cite[Lemma 8.1]{BC}).

\begin{lm}\label{lm1} Let $\mgg=\rmc^0(\mgg)\supset\rmc^1(\mgg)\supset\dots$ and 
$0=\rmc_0(\mgg)\subset\rmc_1(\mgg)\subset\dots$ be, respectively, the 
upper and lower central series of a Lie algebra $\mgg$.
If either $\mgg$ or $\R\oplus \mgg$ is a symplectic Lie algebra then  
\begin{equation} \label{eq.lm1}
\dim\rmc_j(\mgg)+\dim \rmc^j(\mgg)\leq \dim \mgg
\end{equation}
for all $j\ge0$.
\end{lm}

We now assume that $\mn$ is a $k$-step nilpotent Lie algebra, 
that is $\rmc^{k}(\mn)=0$ and $\rmc^{k-1}(\mn)\neq 0$.

Let $X\in\rmc^{i}(\mn)$ and $Y\in\rmc^{j}(\mn)$. If 
$\omega=d\sigma\in\Lambda^2\mn^*$ is an exact 2-form and $i+j\ge k-1$, 
then $[X,Y]\in \rmc^{i+j+1}(\mn)=0$ and thus 
\begin{equation}\label{eq.exact_form}
 \omega(X,Y)=-\sigma([X,Y])=0.
\end{equation}
On the other hand, if 
$\omega\in\Lambda^2\mn^*$ is a closed 2-form and $i+j\ge k$, then \cite[Proposition 8.2]{BC}
\begin{equation}\label{eq.closed_form}
 \omega(X,Y)=0.
\end{equation}

These observations can be expressed as follows. 
For each  $j=1,\dots,k$, let $\mm_j\subset \mn^*$ be subspaces defined inductively so that 
\begin{equation}\label{eq.m_j2}
 \mm_{1} \oplus \mm_{2} \oplus\dots \oplus \mm_{j}=\{f\in\mn^*:f|_{\rmc^{j}(\mn)}=0\}.
\end{equation}
In particular, $\mn^*=\mm_{1} \oplus \mm_{2} \oplus\dots \oplus \mm_{k}$ and thus 
\[
\Lambda^2\mn^* = \bigoplus_{ 1\le i\le j\le k} \mm_i\wedge \mm_j
\]
where $\mm_i\wedge \mm_j$ denotes the skew-symmetrization of $\mm_i\otimes \mm_j$, that is
$2f(X,Y)=f(X,Y)-f(Y,X)$ if $f\in\mm_i\otimes\mm_j$ and $X,Y\in\mn$.

The following lemma generalizes \cite[Lemma 2.8]{BE-GO} and is a straightforward consequence of 
\eqref{eq.exact_form} and \eqref{eq.closed_form}.

\begin{lm} \label{lm.kerim} Let $\mn$ be a $k$-step nilpotent Lie algebra
and let $\mm_j$, $1\le j\le k$, be as in \eqref{eq.m_j2}. Then
\begin{equation}\label{eq.kernel}
\textrm{\rm ker}(d:\Lambda^2 \mn^*\to \Lambda^3\mn^*)\subset 
\bigoplus_{\begin{smallmatrix} i+j\le k+1 \\[1mm] 1\le i\le j\le k\end{smallmatrix}} \mm_i\wedge \mm_j
\end{equation}
and
\begin{equation}\label{eq.exactas}
\textrm{\rm image}(d:\mn^*\to \Lambda^2\mn^*)\subset 
\bigoplus_{\begin{smallmatrix} i+j\le k \\[1mm] 1\le i\le j\le k\end{smallmatrix}} \mm_i\wedge \mm_j.
\end{equation} 
\end{lm}

\section{Reductive Lie algebras acting on symplectic nilpotent Lie algebras}

In this section, $\mn$ is a nilpotent Lie algebra and $\text{Der}(\mn)$
is its Lie algebra of derivations.

Let 
$\mn=\mn^0\supsetneq\mn^1\supsetneq\dots\supsetneq\mn^{k-1}\supsetneq\mn^{k}=0$
be a filtration of $\mn$, that is 
$$
[\mn^i,\mn^j]\subset\mn^{i+j+1},\qquad 0\le i,j\le k,
$$
and let $\mgg$ be a reductive Lie subalgebra of $\text{Der}(\mn)$, 
acting reductively on $\mn$ and preserving the filtration.

As we did with the lower central series, we introduce, for each  $j=1,\dots,k$, 
$\mgg$-submodules $\mm_j\subset \mn^*$ defined inductively so that 
\begin{equation}\label{eq.m_j}
 \mm_{1} \oplus \mm_{2} \oplus\dots \oplus \mm_{j}=\{f\in\mn^*:f|_{\mn^j}=0\}.
\end{equation}
For each $t=1,\dots,\lceil k/2\rceil$ let 
\begin{equation}\label{eq.definition_P}
 P_t:\Lambda^2 \mn^*\to \mm_t\wedge\mm_{k-t+1}
\end{equation}
be the projection with respect to the decomposition 
$\Lambda^2\mn^*= \bigoplus_{1\le i\le j\le k} \mm_i\wedge \mm_j$. 
It is clear that $P_t$ is a $\mgg$-module morphism.

\begin{defi}\label{def.accurate}
We say that a filtration $\mn=\mn^0\supsetneq\mn^1\supsetneq\dots\supsetneq\mn^{k-1}\supsetneq\mn^{k}=0$
is {\em accurate} if $\omega(X,Y)=0$ for all $X\in\mn^i$, $Y\in \mn^j$ in the following two cases:
(i) when $\omega$ is an exact 2-form and $i+j\geq k-1$, 
(ii) when $\omega$ is a closed 2-form and $i+j\geq k$.
\end{defi}

\begin{rem}\label{rmk.accurate} We notice the following facts.
\begin{enumerate}[(1)] 
\item The  lower central series is accurate, see \eqref{eq.exact_form} and \eqref{eq.closed_form}.
\item\label{item.abelian} There are accurate filtrations other than the lower central series: 
for instance if $\mn$ is abelian, the only accurate filtrations are the lower central series and 
$\mn^0=\mn\supsetneq \mn^1\supsetneq \mn^2=0$ with $\dim\mn^1=1$.
Also, if $\mn=\R T\oplus \langle X,Y,Z\rangle $, with $[X,Y]=Z$,  then 
$\mn=\mn^0\supsetneq\langle Z,T\rangle\supsetneq\langle T\rangle\supsetneq\mn^{3}=0$ is
accurate.
\item\label{item.desc} If $\mn=\mn^0\supsetneq\mn^1\supsetneq\dots\supsetneq\mn^{k-1}\supsetneq\mn^{k}=0$
is accurate and $\mm_j$,  $j=1,\dots,k$, are defined as above, 
then \eqref{eq.kernel} and \eqref{eq.exactas} hold.
\item \label{item.accurate}If $\tilde\mn=\R T\oplus\mn $ and $\mn=\mn^0\supsetneq\mn^1\supsetneq\dots\supsetneq\mn^{k-1}\supsetneq\mn^{k}=0$
is an accurate
filtration of a nonabelian Lie algebra $\mn$ such that $\mn^1=\rmc^1(\mn)$, 
then
\[
 \tilde\mn^j=
 \begin{cases}
  \R T \oplus\mn^j,&\text{ if $j=1,\dots,t$;}\\
  \mn^{j},&\text{ if $j=t+1,\dots,k$} 
 \end{cases}
\] 
defines an accurate filtration of $\tilde\mn$ for all $t=1,\dots,\lceil k/2\rceil$, of the same length $k$. 
Indeed, 
on the one hand it easy to check that $\omega(X,Y)=0$ in case (i) of Definition \ref{def.accurate}.
On the other hand, assume that  $\omega$ is a closed 2-form in $\Lambda^2\tilde\mn^*$ and let 
$T^*\in\tmn^*$ be defined by $T^*(T)=1$ and $T^*|_{\mn}=0$. 
Then 
 $\omega=T^*\wedge \xi +\omega_0$ with $\omega_0$ a closed 2-form in $\Lambda^2\mn^*$ and 
 $\xi\in\mn^*$ with $\xi|_{\rmc^1(\mn)}=0$ (in particular $\xi|_{\mn^j}=0$ for $j\ge1$).
Since the filtration of $\mn$ is accurate, it is now straightforward to see that if 
$X\in\tilde\mn^i$, $Y\in\tilde\mn^j$ with $i+j\geq k$, then  $\omega(X,Y)=0$.
\end{enumerate}
\end{rem}

\begin{teo}\label{thm.Criterio2} Let $\mn$ be a  nonabelian nilpotent Lie algebra and let  
$\mn=\mn^0\supsetneq\mn^1\supsetneq\dots\supsetneq\mn^{k-1}\supsetneq\mn^{k}=0$ be an accurate filtration of $\mn$ invariant under the action of a reductive subalgebra $\mgg\subset \Der(\mn)$.
Assume that either 
\begin{enumerate}[a)]
\item $\mn$ is symplectic, or
\item $\R \oplus \mn$ is symplectic and  $\mn^1=\rmc^1(\mn)$.
\end{enumerate}
Then, for any $t=1,\dots,\lceil k/2\rceil$ such that
\begin{equation}\label{eq:condt}
\dim\mn^{k-t} + \dim\mn^{t-1}>\dim \mn,
\end{equation}
(in particular for $t=1$)
there is an irreducible 
$\mgg$-submodule $U$ of $H^2(\mn)$ such that $P_{t}(\sigma)\ne0$ for any $\sigma\in\Lambda^2\mn^*$  corresponding to a nonzero cohomology class in $U$.
\end{teo}

\begin{proof}
We first assume that $\mn$ is symplectic and let 
$\omega$ be a symplectic form. Let $\mm_j$, $1\le j\le k$, be as in \eqref{eq.m_j}.
Since the filtration is accurate, it follows from Remark \ref{rmk.accurate}(\ref{item.desc}) that
\begin{equation}\label{eq:wij}
\omega
=\sum_{\begin{smallmatrix} i+j\le k+1 \\[1mm] 1\le i\le j\le k\end{smallmatrix}}\omega_{i,j}, 
\quad 
\omega_{i,j}\in\mm_i\wedge \mm_j.
\end{equation}
Recall that the map $\psi:\mn\lra\mn^*$ defined as $\psi(X)=\omega(X,\;\cdot\;)$ is an isomorphism. 
Let $t=1,\dots,\lceil k/2\rceil$ be as in the hypothesis and let $X \in\mn^{k-t}$.
Then 
\begin{equation*}
\psi(X)= 
\sum_{\begin{smallmatrix} i+j\le k+1 \\[1mm] k-t< i\le j\le k \end{smallmatrix}}\omega_{i,j}(X,\;\cdot\;)\;- 
\sum_{\begin{smallmatrix} i+j\le k+1 \\[1mm] 1\le i\le j\le k,\; k-t< j\end{smallmatrix}}\omega_{i,j}(\;\cdot\;,X),
\end{equation*}
since the annihilator of $\mn^{k-t}$ is $\mm_1\oplus\mm_2\oplus\cdots\oplus\mm_{k-t}$.
Since $t\le k/2$, the first sum is empty. 
In the second sum $j\ge k+1-t$ and thus
\begin{equation*}
\psi(X)= - 
\sum_{j=k+1-t}^k\sum_{i=1}^{t}\omega_{i,j}(\;\cdot\;,X).
\end{equation*}
If in addition  $Y\in\mn^{t-1}$, $\omega_{i,j}(Y,X)=0$ for all $i\leq t-1$ and therefore
\begin{equation*}
\psi(X)(Y)= - \sum_{j=k+1-t}^k\sum_{i=t}^{k+1-j}\omega_{i,j}(Y,X).
\end{equation*}
The only pair $(i,j)$ satisfying $t\le i \le k+1-j$ and $k+1-t \le j \le k$ is $(i,j)=(t,k+1-t)$ and 
therefore
\begin{equation}\label{eq:wt}
\psi(X)(Y)=-\omega_{t,k+1-t}(Y,X) \quad \text{for all }X \in\mn^{k-t},\; Y\in\mn^{t-1}.
\end{equation}
We will show now that $P_{t}(\omega)\ne0$.
Assume, on the contrary, that $\omega_{t,k+1-t}=P_{t}(\omega)=0$. Then \prettyref{eq:wt} implies that
\[
\begin{array}{rl}
\omega(X,\;\cdot\;)=0, &\text{ if } t=1;\\
\psi(\mn^{k-t})\subset\mm_1\oplus\mm_2\oplus\cdots\oplus\mm_{t-1},  &\text{ if } t\geq 2.
\end{array} 
\]
The case corresponding to $t=1$ is a contradiction since $\omega$ is nondegenerate. 
In the case $t\geq 2$, since $\psi$ is injective, one would obtain
$$\dim \mn^{k-t}
\leq \dim(\mm_{1}+\mm_2+\ldots +\mm_{t-1})
= \dim\mn -\dim\mn^{t-1}$$ 
which is a contradiction to \prettyref{eq:condt}. 
Thus, we have proved that $P_t(\omega)\neq 0$.

Let $V$ be the $\mgg$-submodule of $\Lambda^2\mn^*$  generated by $\omega$.
Since $\omega$ is closed, every element in $V$ is closed. 
Let 
\[
 V=V_1\oplus \dots \oplus V_r
\]
be its $\mgg$-module decomposition into irreducible submodules. 
If $\omega=\omega_1+ \dots + \omega_r$, with $\omega_i\in V_i$, we may assume that $P_t(\omega_{1})\ne0$
and thus 
\[
 P_t|_{V_{1}}:V_{1}\to \mm_t\wedge\mm_{k-t+1}
\]
is an injective $\mgg$-module morphism. Again, every element in $V_{1}$ is closed.  

If $\sigma\in V_{1}$ and $\sigma\ne0$, then $P_t(\sigma)\ne0$. 
Since (see Remark \ref{rmk.accurate}(\ref{item.desc}))
\[
\text{image}(d:\mn^*\to \Lambda^2\mn^*) \cap  \mm_t\wedge\mm_{k-t+1}=0,
\]
it follows that the cohomology class corresponding to $\sigma$ is nonzero.
Therefore, if $\pi:\ker(d:\Lambda^2 \mn^*\to \Lambda^3 \mn^*)\to H^2(\mn)$ is the canonical projection, then $\pi|_{V_1}$ is injective 
and $U=\pi(V_1)$ satisfies the required condition.

We now assume that $\tilde\mn=\R T \oplus\mn$ is symplectic and $\mn^1=\rmc^1(\mn)\ne0$.
Let $t=1,\dots,\lceil k/2\rceil$ be as in the hypothesis,
then it follows from Remark \ref{rmk.accurate}(\ref{item.accurate}) that the filtration  of $\tmn$ given by 
\[
 \tilde\mn^j=
 \begin{cases}
  \R T \oplus\mn^j,&\text{ if $j=1,\dots,t$;}\\
  \mn^{j},&\text{ if $j=t+1,\dots,k$} .
 \end{cases}
\] 
is accurate and it is also preserved by $\mgg$ viewed as a reductive Lie subalgebra of $\text{Der}(\tilde\mn)$ (acting by zero on $T$). 
If $T^*\in\tilde\mn^*$ is defined by $T^*(T)=1$ and $T^*|_{\mn}=0$.
then we may choose
\[
 \tilde\mm_j=
 \begin{cases}
  \R T^* \oplus\mm_t,&\text{ if $j=t$;}\\
  \mm_{j},&\text{ if $j\ne t$.} 
 \end{cases}
\]
Let  
\[
\tilde P_t:\Lambda^2 \tmn^*\to \tmm_t\wedge\tmm_{k-t+1}=T^*\wedge\mm_{k-t+1}\;\oplus\; \mm_t\wedge\mm_{k-t+1}.
\]
It follows from \prettyref{eq:condt} that $ \dim\tmn^{k-t} + \dim\tmn^{t-1}>\dim \tmn$
and hence  there is an irreducible $\mgg$-submodule $ U$ of $H^2(\tilde\mn)$ such that 
$\tilde P_t(\sigma)\ne0$ for any $\sigma\in\Lambda^2\tilde\mn^*$ corresponding to a nonzero cohomology 
class in $U$.
Since
\[
 H^2(\tilde\mn)=H^2(\mn)\;\oplus\; T^*\wedge H^1(\mn)\simeq H^2(\mn)\;\oplus\; T^*\wedge \mm_1
\]
as $\mgg$-modules and $\tilde P_t( T^*\wedge \mm_1)=0$, it follows that $U$ is a submodule of $H^2(\mn)$, 
and now $P_t(\sigma)=\tilde P_t(\sigma)\ne0$ 
for any $\sigma\in\Lambda^2\mn^*$ corresponding to a nonzero cohomology 
class in $U$.
\end{proof}

\begin{rem} Theorem 3.1. in \cite{dB3} is now a consequence of the theorem above. 
Indeed, the spectral sequence limit term $E_\infty^{0,2}$ in \cite{dB3} 
corresponds to the image of projection $P_1$ in Theorem \ref{thm.Criterio2}, when considering $\mgg$ as the 
trivial subalgebra of $\text{Der}(\mn)$ and the filtration given by the lower central series of $\mn$.
\end{rem} 

An immediate consequence of Theorem \ref{thm.Criterio2} and Remark \ref{rmk.accurate}(1) is the following corollary.

\begin{cor}\label{cor.seriecent}
Let $\mn$ be a nonabelian $k$-step nilpotent Lie algebra such that either $\mn$ or $\R\oplus\mn$ is 
symplectic and let $\mgg$ be a reductive subalgebra of $\Der(\mn)$ (a possible instance is $\mg=0$). 
Then, for any $t=1,\dots,\lceil k/2\rceil$ such that 
\begin{equation}\label{eq.seriecent} 
\dim \rmc^{k-t}(\mn) + \dim \rmc^{t-1}(\mn)>\dim \mn,
\end{equation}
(in particular for $t=1$) there is an irreducible 
$\mgg$-submodule $U$ of $H^2(\mn)$ such that $P_{t}(\sigma)\ne0$ for any $\sigma\in\Lambda^2\mn^*$ 
corresponding to a nonzero cohomology class in $U$.
\end{cor}

\section{Nilradicals of parabolic subalgebras with symplectic structures}\label{sec.4}
In this section we study the existence of symplectic structures on 
the nilradicals $\mn$ (or their trivial one-dimensional extensions) 
of parabolic subalgebras $\mpp\simeq \mgg_1\ltimes \mn$ of split real simple Lie algebras. 
We use Corollary \ref{cor.seriecent} applied to the 
reductive Lie algebra $\mgg_1$ viewed as a subalgebra of $\text{Der}(\mn)$.

\subsection{Basic facts about parabolic subalgebras}
Let $\mgg$ be a split real simple 
Lie algebra with triangular decomposition $\mgg=\mgg^{-}+\mh+\mgg^+$ associated to 
a positive root system $\Delta^+$ with simple roots $\Pi$ and let 
$\mb=\mh+\mgg^+$ be the corresponding Borel subalgebra of $\mgg$.
Let $W$ be the associated Weyl group of $\mgg$. 
Given $\alpha\in\Delta$, let $s_{\alpha}\in W$ denote the associated reflection and, for $w\in W$,
let $\ell(w)$ be the length of $w$ \cite{Hu}[\S10.3].
As usual, if $\gamma\in\Delta$, then $X_\gamma$ denotes an arbitrary root vector in the root space $\mgg_\gamma$.
Given $\gamma\in\Delta$ and $\alpha\in\Pi$ let $\coord_{\alpha}(\gamma)$ denote the $\alpha$-coordinate of $\gamma$ 
when it is expressed as a linear combination of simple roots. Let $\gamma_{\max}$ denote the unique maximal root of $\Delta^+$. 

The set of parabolic Lie subalgebras of $\mgg$ containing $\mb$ is parametrized by 
the subsets of the set of simple roots $\Pi$ as follows.
Given a subset $\Pi_0\subset\Pi$, the corresponding parabolic subalgebra of $\mgg$
is $\mpp\simeq \mgg_1\ltimes \mn$ where
$\mgg_1=\mgg_1^-\oplus \mh\oplus \mgg_1^+$ and  
\begin{eqnarray}
\Delta_1^+ &=& \{\gamma\in\Delta^+: \coord_{\alpha}(\gamma)=0\text{ for all }\alpha\in\Pi_0\},\nonumber\\
\Delta_{\mn}^+ &=& \{\gamma\in\Delta^+: \coord_{\alpha}(\gamma)\neq0\text{ for some }\alpha\in\Pi_0
\},\nonumber\\
\mgg_1^+&=& \mbox{ subspace of $\mgg^+$ spanned by the vectors }X_\gamma \mbox{  with } \gamma\in \Delta_1^+,\nonumber\\
\mgg_1^-&=& \mbox{ subspace of $\mgg^-$ spanned by the vectors }X_{-\gamma} \mbox{  with } \gamma\in \Delta_1^+,\nonumber\\
\mn&=& \mbox{ subspace of $\mgg^+$ spanned by the vectors }X_\gamma \mbox{  with } \gamma\in \Delta_{\mn}^+.\nonumber
\end{eqnarray}
 
On the one hand, $\mgg_1$ is reductive and it can be viewed as
a Lie subalgebra of $\text{Der}(\mn)$ via $\ad_{\mgg}$. 
Let $\mb_1=\mh+\mgg_1^+$ be the Borel subalgebra of $\mgg_1$ associated to $\mb$.

On the other hand, $\mn$ is nilpotent and its lower central series (which coincides, after transposing the 
indexes, with the upper central series) can be described as follows \cite[Thm 2.12]{Ko1}. Given $\gamma\in\Delta$, let 
\[
 o(\gamma)=\sum_{\alpha\in\Pi_0} \coord_{\alpha}(\gamma)
\]
(in particular $\gamma\in\Delta_{\mn}^+$ if and only if $ o(\gamma)>0$) and, for $i\in\Z$, let 
\begin{equation}\label{eq.g_{()}}
 \mg_{(i)}=\bigoplus_{\begin{smallmatrix}
               \gamma\in\Delta \\
               o(\gamma)=i
              \end{smallmatrix}} \mg_{\gamma}.
\end{equation}
If  
$\mn=\rmc^0(\mn)\supset\rmc^1(\mn)\supset\dots\supset\rmc^{k-1}(\mn)\supset\rmc^{k}(\mn)=0$ is
the lower central series of $\mn$, then
\begin{equation}\label{eq.lcs}
 \rmc^j(\mn)=\bigoplus_{i=j+1}^{k}\mg_{(i)},\;\; k=o(\gamma_{\text{max}})
 \text{ and $\rmc^{k-1}(\mn)=\mg_{(k)}$ is the center of $\mn$}.
\end{equation}
\begin{rem}\label{rem.abelian} It follows from this description of the lower central series that 
the nilradical $\mn$ is abelian if and only if $\Pi_0=\{\alpha\}$ and $\coord_{\alpha}(\gamma_{\max})=1$.
\end{rem}

If $(\cdot,\cdot)$ is the Killing form of $\mgg$, then $\mn^*$ is identified with 
\begin{align*}
  \mn^-
& = \mbox{ subspace of $\mgg^-$ spanned by the vectors }X_{-\gamma} \mbox{  with } \gamma\in \Delta_{\mn}^+ \\
& =\bigoplus_{i=1}^{k}\mg_{(-i)}.
\end{align*}
Under this identification the dual action of $\mgg_1$ on $\mn^*$ becomes the adjoint action of $\mgg_1$ on $\mn^-$.
Also, if $j=1,\dots,k$ and 
$\mm_j$ is defined as in \eqref{eq.m_j}, then 
\begin{equation}\label{eq.mm_j=mg_(-j)}
\mm_j\simeq\mg_{(-j)}.
\end{equation}
In particular, if $P=P_1$ is the projection considered in \eqref{eq.definition_P},
then $P$ becomes
\begin{equation}\label{eq.definition_PP}
 P:\Lambda^2 \mn^-\to \mg_{(-1)}\wedge\mg_{(-k)}.
 \end{equation}

\subsection{The 2-cohomology of the nilradicals of parabolic subalgebras}

A well known result of Kostant \cite{KO} describes the irreducible $\mgg_1$-submodules of $H^2(\mn)$.
In fact, up to scalars, the set of representatives of the $\mb_1$-highest weight vectors in $H^2(\mn)$ is:
\begin{equation}
H^2(\mn)_{\text{hwv}}=
\{X_{-\gamma_1}\wedge X_{-\gamma_{2}}\in\Lambda^2\mn^-:\{\gamma_1,\gamma_2\}=w\Delta^-\cap\Delta^+,\,w\in W^{1,2}\}
\end{equation}
where 
\[
 W^{1,2}=\{w\in W : w\Delta^-\cap\Delta^+\subset\Delta^+_{\mn}\text{ and }\ell(w)=2\}.
\]
The following theorem describes $H^2(\mn)_{\textrm{\rm hwv}}$ more precisely.
\begin{teo}\label{prop.H2_hwv} A set of representatives of the $\mb_1$-highest weight vectors in $H^2(\mn)$ is
\begin{align*}
H^2(\mn)_{\rm hwv}
&=\{ X_{-\alpha}\wedge X_{-s_{\alpha}(\beta)} : 
\alpha\in\Pi_0,\;\beta\in\Pi \mbox{ and }s_{\alpha}(\beta)\in \Delta_{\mn}^+\} \\
&=\{ X_{-\alpha}\wedge X_{-\beta} :\alpha,\beta\in\Pi_0,\,(\alpha,\beta)=0\} 
\,\cup \\
&\hspace{2cm}\{ X_{-\alpha}\wedge X_{-s_{\alpha}(\beta)} : \alpha\in\Pi_0,\;\beta\in\Pi,\;(\beta,\alpha)<0\}. 
 \end{align*}
\end{teo}
\begin{proof}
If $w\in W$ and $\ell(w)=2$, then $w=s_\alpha s_\beta$ for some $\alpha,\beta\in\Pi$, $\alpha\ne\beta$.
Taking into account that $w\Delta^-\cap\Delta^+$ contains exactly two elements \cite{Hu}[\S10.3],
  a direct computation shows that
$ w\Delta^-\cap\Delta^+=\{ \alpha,s_\alpha(\beta) \}$.

Now $w\in W^{1,2}$ if and only if $w\Delta^-\cap\Delta^+=\{ \alpha,s_\alpha(\beta) \}\subset\Delta_{\mn}^+$. 
Since $\alpha$ and $\beta$ are simple and 
 \begin{equation}\label{eq:sigma}
 s_{\alpha}(\beta)=\beta-2\frac{(\beta,\alpha)}{||\alpha||^2} \alpha ,
 \end{equation}
it follows that $w\in W^{1,2}$ if and only if $\alpha\in\Pi_0$ and, in the case $(\beta,\alpha)=0$, $\beta\in\Pi_0$.
 \end{proof}

\subsection{Symplectic nilradicals}

\begin{pro}\label{cor.centro}
Let $\mn$ be a nonabelian nilradical of the parabolic subalgebra of $\mgg$ associated to $\Pi_0$.
If either $\mn$ or $\R\oplus\mn$ is symplectic, 
then there exist $\alpha\in\Pi_0$ and $\beta\in\Pi$, 
such that  $\mgg_{s_\alpha(\beta)}\subset \rmc^{k-1}(\mn)$  and  $(\beta,\alpha)<0$.
\end{pro}
\begin{proof}
If either $\mn$ or $\R\oplus\mn$ is symplectic, 
we know from case $t=1$ in Corollary \ref{cor.seriecent} that there exists 
a nonzero  $\mb_1$-highest weight vector $u\in H^2(\mn)_{\text{hwv}}$
such that $P(\sigma)\ne0$ for any representative $\sigma$ of $u$.  
Let $\alpha\in\Pi_0$ and $\beta\in\Pi$ be as in Theorem \ref{prop.H2_hwv}
the roots corresponding to $u$.
Since $X_{-\alpha}\in\mg_{(-1)}$, it follows from \eqref{eq.definition_PP} that 
\[
P(X_{-\alpha}\wedge X_{-s_{\alpha}(\beta)})\ne 0
\]
if and only if $X_{-s_{\alpha}(\beta)}\in\mg_{(-k)}$, or 
if and only if  $\mgg_{s_\alpha(\beta)}\subset \rmc^{k-1}(\mn)$.
Since $\mn$ is nonabelian, it follows that $k>1$ and hence $s_{\alpha}(\beta)\ne\beta$.  
This shows that $(\beta,\alpha)\ne0$ and being both simple roots, we obtain $(\beta,\alpha)<0$.
\end{proof}

\begin{pro}\label{pro:pro3}
Given $\gamma\in\Delta_\mn^+$, the root space $\mgg_\gamma$ is contained in $\rmc^{k-1}(\mn)$
if and only if $\coord_{\alpha}(\gamma)=\coord_{\alpha}(\gamma_{\max})$ for all $\alpha\in\Pi_0$.
\end{pro}
\begin{proof} If $\coord_{\alpha}(\gamma)=\coord_{\alpha}(\gamma_{\max})$ 
for all $\alpha\in\Pi_0$ then it is clear  from \eqref{eq.lcs} that  
$\mgg_\gamma$ is contained in the center of $\mn$.

Conversely, let $\gamma\in\Delta_\mn^+$ be such that  $\mgg_\gamma\subset\rmc^{k-1}(\mn)$ 
and assume that $\coord_{\alpha}(\gamma)<\coord_{\alpha}(\gamma_{\max})$ 
for some $\alpha\in\Pi_0$.
Since $o(\gamma_{\max})=o(\gamma)=k-1$,  
there exists some $\beta\in\Pi_0$ such that $\coord_{\beta}(\gamma)>\coord_{\beta}(\gamma_{\max})$ which is clearly  not possible. 
%
\end{proof}

\begin{pro}\label{pro:pro4}
Let $\mn$ be a nonabelian nilradical of the parabolic subalgebra of $\mgg$ associated to $\Pi_0$.
If either 
$\mn$ or $\R\oplus\mn$ is symplectic, then one of the following statements holds
\begin{enumerate}
\item $\Pi_0=\{\alpha\}$ and there exists $\beta\in\Pi$ such that 
$\coord_{\alpha}(\gamma_{\max})=-2\frac{(\beta,\alpha)}{||\alpha||^2}$.
\item $\Pi_0=\{\alpha,\beta\}$, 
$\coord_{\beta}(\gamma_{\max})=1$  and  
$\coord_{\alpha}(\gamma_{\max})=-2\frac{(\beta,\alpha)}{||\alpha||^2}$.
\end{enumerate}  
In any case, $(\beta,\alpha)<0$.
\end{pro}

\begin{proof} 
If either 
$\mn$ or $\R\oplus\mn$ is symplectic, then it follows 
from Propositions \ref{cor.centro} and \ref{pro:pro3}
that there exist $\alpha\in\Pi_0$ and $\beta\in\Pi$, 
such that 
\[
\coord_{\alpha'}(s_\alpha(\beta))=\coord_{\alpha'}(\gamma_{\max}) 
\]
for all 
$\alpha'\in\Pi_0$.
Since 
 \begin{equation*}
 s_{\alpha}(\beta)=\beta-2\frac{(\beta,\alpha)}{||\alpha||^2}\,\alpha,
 \end{equation*}
 and $\alpha$ and $\beta$ are simple,  $s_{\alpha}(\beta)$ is a linear combination of at most two simple roots.
 Since $\coord_{\alpha'}(\gamma_{\max})\ge1$ for any simple root $\alpha'$ it follows that 
 $\Pi_0$ contains, in addition to $\alpha$, at most $\beta$.
 
 If $\Pi_0=\{\alpha\}$ then we must have 
$\coord_{\alpha}(\gamma_{\max})=-2\frac{(\beta,\alpha)}{||\alpha||^2}$. If instead  $\Pi_0=\{\alpha,\beta\}$, then we must have 
$\coord_{\beta}(\gamma_{\max})=1$  and 
$\coord_{\alpha}(\gamma_{\max})=-2\frac{(\beta,\alpha)}{||\alpha||^2}$. 
Again, since $\coord_{\alpha}(\gamma_{\max})\ge 1$, it follows that $(\beta,\alpha)<0$.
\end{proof}

\smallskip

The following table shows the Dynkin diagram of each simple Lie algebra and the labels indicate
the coordinate of the maximal root in the corresponding simple root. 

\setlength{\unitlength}{0.7cm}
\begin{picture}(6,1.2)(-3,-.3)
\put(-3,-0.15){$A_n,\,n\ge1:$}
\put(0,0){\circle{.2}}
\put(0.3,0){\line(1,0){1}}
\put(1.6,0){\circle{.2}}
\put(1.9,0){\line(1,0){1}}
\put(3.2,0){$\dots$}
\put(4.0,0){\line(1,0){1}}
\put(5.3,0){\circle{.2}}
\put(5.6,0){\line(1,0){1}}
\put(6.9,0){\circle{.2}}
\put(7.2,0){\line(1,0){1}}
\put(8.5,0){\circle{.2}}
\put(-0.1,-.4){\scriptsize 1}
\put(1.5,-.4){\scriptsize 1}
\put(5.2,-.4){\scriptsize 1}
\put(6.8,-.4){\scriptsize 1}
\put(8.4,-.4){\scriptsize 1}
\end{picture}

\setlength{\unitlength}{0.7cm}
\begin{picture}(6,1.2)(-3,-.3)
\put(-3,-0.15){$B_n,\,n\ge2:$}
\put(0,0){\circle{.2}}
\put(0.3,0){\line(1,0){1}}
\put(1.6,0){\circle{.2}}
\put(1.9,0){\line(1,0){1}}
\put(3.2,0){$\dots$}
\put(4.0,0){\line(1,0){1}}
\put(5.3,0){\circle{.2}}
\put(5.6,0){\line(1,0){1}}
\put(6.9,0){\circle{.2}}
\put(7.2,0.05){\line(1,0){1}}
\put(7.2,-0.05){\line(1,0){1}}
\put(7.6,-0.13){$>$}
\put(8.5,0){\circle{.2}}
\put(-0.1,-.4){\scriptsize 1}
\put(1.5,-.4){\scriptsize 2}
\put(5.2,-.4){\scriptsize 2}
\put(6.8,-.4){\scriptsize 2}
\put(8.4,-.4){\scriptsize 2}
\end{picture}

\setlength{\unitlength}{0.7cm}
\begin{picture}(6,1.2)(-3,-.3)
\put(-3,-0.15){$C_n,\,n\ge3:$}
\put(0,0){\circle{.2}}
\put(0.3,0){\line(1,0){1}}
\put(1.6,0){\circle{.2}}
\put(1.9,0){\line(1,0){1}}
\put(3.2,0){$\dots$}
\put(4.0,0){\line(1,0){1}}
\put(5.3,0){\circle{.2}}
\put(5.6,0){\line(1,0){1}}
\put(6.9,0){\circle{.2}}
\put(7.2,0.05){\line(1,0){1}}
\put(7.2,-0.05){\line(1,0){1}}
\put(7.6,-0.13){$<$}
\put(8.5,0){\circle{.2}}
\put(-0.1,-.4){\scriptsize 2}
\put(1.5,-.4){\scriptsize 2}
\put(5.2,-.4){\scriptsize 2}
\put(6.8,-.4){\scriptsize 2}
\put(8.4,-.4){\scriptsize 1}
\end{picture}

\setlength{\unitlength}{0.7cm}
\begin{picture}(6,2.5)(-3,-.3)
\put(-3,-0.15){$D_n,\,n\ge4:$}
\put(0,0){\circle{.2}}
\put(0.3,0){\line(1,0){1}}
\put(1.6,0){\circle{.2}}
\put(1.9,0){\line(1,0){1}}
\put(3.2,0){$\dots$}
\put(4.0,0){\line(1,0){1}}
\put(5.3,0){\circle{.2}}
\put(5.6,0){\line(1,0){1}}
\put(6.9,0){\circle{.2}}
\put(7.2,0){\line(1,0){1}}
\put(8.5,0){\circle{.2}}
\put(6.9,0.2){\line(0,1){1}}
\put(6.9,1.5){\circle{.2}}
\put(-0.1,-.4){\scriptsize 1}
\put(1.5,-.4){\scriptsize 2}
\put(5.2,-.4){\scriptsize 2}
\put(6.8,-.4){\scriptsize 2}
\put(6.5,1.4){\scriptsize 1}
\put(8.4,-.4){\scriptsize 1}
\end{picture}

\setlength{\unitlength}{0.7cm}
\begin{picture}(6,2.5)(-3,-.3)
\put(-3,-0.15){$E_6:$}
\put(0,0){\circle{.2}}
\put(0.3,0){\line(1,0){1}}
\put(1.6,0){\circle{.2}}
\put(1.9,0){\line(1,0){1}}
\put(3.3,0){\circle{.2}}
\put(3.6,0){\line(1,0){1}}
\put(4.9,0){\circle{.2}}
\put(5.2,0){\line(1,0){1}}
\put(6.5,0){\circle{.2}}
\put(3.3,0.2){\line(0,1){1}}
\put(3.3,1.5){\circle{.2}}
\put(-0.1,-.4){\scriptsize 1}
\put(1.5,-.4){\scriptsize 2}
\put(3.2,-.4){\scriptsize 3}
\put(4.8,-.4){\scriptsize 2}
\put(2.9,1.4){\scriptsize 2}
\put(6.4,-.4){\scriptsize 1}
\end{picture}

\setlength{\unitlength}{0.7cm}
\begin{picture}(6,2.5)(-3,-.3)
\put(-3,-0.15){$E_7:$}
\put(0,0){\circle{.2}}
\put(0.3,0){\line(1,0){1}}
\put(1.6,0){\circle{.2}}
\put(1.9,0){\line(1,0){1}}
\put(3.3,0){\circle{.2}}
\put(3.6,0){\line(1,0){1}}
\put(4.9,0){\circle{.2}}
\put(3.3,0.2){\line(0,1){1}}
\put(3.3,1.5){\circle{.2}}
\put(5.2,0){\line(1,0){1}}
\put(6.5,0){\circle{.2}}
\put(6.8,0){\line(1,0){1}}
\put(8.1,0){\circle{.2}}
\put(-0.1,-.4){\scriptsize 2}
\put(2.9,1.4){\scriptsize 2}
\put(1.5,-.4){\scriptsize 3}
\put(3.2,-.4){\scriptsize 4}
\put(4.8,-.4){\scriptsize 3}
\put(6.4,-.4){\scriptsize 2}
\put(8.0,-.4){\scriptsize 1}
\end{picture}

\setlength{\unitlength}{0.7cm}
\begin{picture}(6,2.5)(-3,-.3)
\put(-3,-0.15){$E_8:$}
\put(0,0){\circle{.2}}
\put(0.3,0){\line(1,0){1}}
\put(1.6,0){\circle{.2}}
\put(1.9,0){\line(1,0){1}}
\put(3.3,0){\circle{.2}}
\put(3.6,0){\line(1,0){1}}
\put(4.9,0){\circle{.2}}
\put(3.3,0.2){\line(0,1){1}}
\put(3.3,1.5){\circle{.2}}
\put(5.2,0){\line(1,0){1}}
\put(6.5,0){\circle{.2}}
\put(6.8,0){\line(1,0){1}}
\put(8.1,0){\circle{.2}}
\put(8.4,0){\line(1,0){1}}
\put(9.7,0){\circle{.2}}
\put(-0.1,-.4){\scriptsize 2}
\put(2.9,1.4){\scriptsize 3}
\put(1.5,-.4){\scriptsize 4}
\put(3.2,-.4){\scriptsize 6}
\put(4.8,-.4){\scriptsize 5}
\put(6.4,-.4){\scriptsize 4}
\put(8.0,-.4){\scriptsize 3}
\put(9.6,-.4){\scriptsize 2}
\end{picture}

\setlength{\unitlength}{0.7cm}
\begin{picture}(6,1.2)(-3,-.3)
\put(-3,-0.15){$F_4:$}
\put(0,0){\circle{.2}}
\put(0.3,0){\line(1,0){1}}
\put(1.6,0){\circle{.2}}
\put(2.0,0.05){\line(1,0){1}}
\put(2.0,-0.05){\line(1,0){1}}
\put(2.4,-0.13){$<$}
\put(3.3,0){\circle{.2}}
\put(3.6,0){\line(1,0){1}}
\put(4.9,0){\circle{.2}}
\put(-0.1,-.4){\scriptsize 2}
\put(1.5,-.4){\scriptsize 3}
\put(3.2,-.4){\scriptsize 4}
\put(4.8,-.4){\scriptsize 2}
\end{picture}

\setlength{\unitlength}{0.7cm}
\begin{picture}(6,1.2)(-3,-.3)
\put(-3,-0.15){$G_2:$}
\put(0,0){\circle{.2}}
\put(0.3,0){\line(1,0){1}}
\put(0.3,0.07){\line(1,0){1}}
\put(0.3,-0.07){\line(1,0){1}}
\put(0.6,-0.13){$>$}
\put(1.6,0){\circle{.2}}
\put(-0.1,-.4){\scriptsize 3}
\put(1.5,-.4){\scriptsize 2}
\end{picture}


\medskip

All the subsets $\Pi_0\subset \Pi$ giving rise to symplectic nilradicals are characterized by 
Remark \ref{rem.abelian} and Proposition \ref{pro:pro4}. These sets can be identified by reading the Dynkin diagrams.
After a careful inspection of them, 
we obtain the possible $\Pi_0$ for each family and we list them in Table 1.
We assume that $\Pi=\{\gamma_1,\dots,\gamma_n\}$ where 
the numbering of simple roots is as \cite{Hu}[Theorem 11.4]
(and is in correspondence with the Dynkin diagrams given above).

\medskip

\begin{center}
\begin{tabular}{l|l|l}
Split real  & \\[-1mm]
 form of type & $\Pi_0=\{\alpha\}$  & $\Pi_0=\{\alpha,\beta\}$ \\
 \hline\hline
 $A_n,\,n\ge1$  & $\alpha= \gamma_k$, $1\leq  k \le  n$ & $\alpha= \gamma_k$, $\beta= \gamma_{k+1}$, $1 \leq  k \le  n-1$  \\
 \hline
  $B_n,\,n\ge2$ & $\alpha= \gamma_1$, $\gamma_n$ &  $\alpha= \gamma_{n}$, $\beta= \gamma_{n-1}$, $n=2$  \\
  \hline
  $C_n,\,n\ge3$ & $\alpha= \gamma_{n-1}$, $\gamma_n$ &  $\alpha= \gamma_{n-1}$, $\beta= \gamma_{n}$ \\
  \hline
  $D_n,\,n\ge4$ & $\alpha=  \gamma_1$,  $ \gamma_{n-1}$,  $ \gamma_n$ & $\alpha= \gamma_{n-1}$, $\beta= \gamma_n$\\
  \hline
 $E_6$ & $\alpha=  \gamma_1$, $ \gamma_6$ & $\emptyset$  \\
 \hline
  $E_7$ & $\alpha=  \gamma_7$ & $\emptyset$  \\
  \hline
   $E_8$ & $\emptyset$ & $\emptyset$  \\
   \hline
   $F_4$ & $\emptyset$ & $\emptyset$  \\
 \hline
 $G_2$  & $\alpha= \gamma_1$ & $\emptyset$  \\ 
 \hline
\end{tabular}

\smallskip

Table 1: $\Pi_0\subset \Pi$ corresponding to abelian nilradicals or satisfying conditions in Proposition \ref{pro:pro4}.  
\end{center}

\smallskip

Note that not every nilpotent Lie algebra arising from a $\Pi_0$ listed above is necessarily symplectic.
The rest of the paper is devoted to distinguish 
which of the candidates in Table 1 give symplectic nilradicals. 

\begin{rem} \label{rem.abelsymp} Note that if $\mgg$ is of type $A$, $D$ or $E$ and $\mn$ is the nilradical associated to some $\Pi_0=\{\alpha\}$ in Table 1, then $\mn$ is abelian since $\coord_{\alpha}(\gamma_{\max})=1$ in any case (see Remark \ref{rem.abelian}). Thus either $\mn$ or $\R\oplus \mn$ is always symplectic. The same assertion is valid for $\mgg$ of type $B$ and $\Pi_0=\{\gamma_1\}$, and $\mgg$ of type $C$ and $\Pi_0=\{\gamma_n\}$.
\end{rem}

The next result summarizes the nilradicals admitting symplectic structures.
\begin{teo}\label{thm.class}
Let $\mn$ be the nilradical of the parabolic subalgebra of $\mgg$ associated to $\Pi_0$ and assume that $\R\oplus \mn$ or $\mn$ is symplectic. Then $\mgg$ and $\Pi_0$ are one of the following:
\begin{enumerate}
\item $\mgg$ is of type $A$, $D$ or $E$ and $\Pi_0=\{\alpha\}$ as in Table 1, resulting $\mn$ abelian;
\item $\mgg$ is of type $B_n$ and $\Pi_0=\{\gamma_1\}$ or  $\mgg$ is of type $C_n$ and $\Pi_0=\{\gamma_n\}$, resulting $\mn$ abelian;
\item $\mgg$ is of type $A_n$ and $\Pi_0=\{\gamma_k,\gamma_{k+1}\}$ where $k$ is either 
$1$ or $n-1$ and $n\geq 1$, or $k=2$ and $n=4$;
\item $\mgg$ is of type $B_n$ and $\Pi_0=\{\gamma_n\}$ for $n=2$ or $n=3$;
\item $\mgg$ is of type $B_2$ and $\Pi_0=\{\gamma_1,\gamma_2\}$;
\item $\mgg$ is of type $C_n$ and $\Pi_0=\{\gamma_{n-1}\}$ for $n=3$ or $n=4$;
\item $\mgg$ is of type $D_n$ and $\Pi_0=\{\gamma_{n-1},\gamma_{n}\}$ for $n=4$.
\end{enumerate}
\end{teo}

\begin{proof} Throughout this proof, given a basis $\mathcal B$ of a Lie algebra, we will denote with
upper indexes the elements of the dual basis of $\mathcal B$. 

The first and second statements are direct consequences of Remark \ref{rem.abelsymp}. 
We consider now the remaining cases $\Pi_0=\{\alpha\}$ in types $B$, $C$ and $G$.
   
\noindent $\bullet$ $B_n$: Here 
 $\Pi=\{\epsilon_1-\epsilon_{2},\dots, \epsilon_{n-1}-\epsilon_{n},\epsilon_n\}$,  $\Pi_0=\{\epsilon_n\}$ 
and thus
$$
\Delta_\mn^+=
\{\epsilon_i: 1\le i\le n\}
\cup
\{\epsilon_j+\epsilon_k: 1\leq j<k\leq n\}.
$$ 
This shows that 
$\mn$ is the free 2-step nilpotent Lie algebra on $n$ generators $\mn_{n,2}$. 
We have $\rmc_1(\mn_{n,2})=\rmc^1(\mn_{n,2})$, with dimension $n(n-1)/2$, and $\dim \mn_{n,2}=n+n(n-1)/2$. 
Since the inequality
$\dim \rmc_1(\mn_{n,2})+\dim\rmc^1(\mn_{n,2})=n(n-1)\leq \dim \mn_{n,2}$ 
holds only if $n\leq 3$, 
it follows from Lemma \ref{lm1} that neither $\mn_{n,2}$ nor $\R\oplus\mn_{n,2}$
admit symplectic structures for $n\ge 4$. 
For $n=3$,  $\mn_{3,2}$ is six-dimensional 
and it is symplectic (see for instance \cite{dB1,DT}). 
On the other hand, $\R\oplus \mn_{1,2}$ 
is abelian of dimension 2 and $\R\oplus\mn_{2,2}$ is a trivial extension of the Heisenberg Lie algebra, 
both well known symplectic Lie algebras.

\medskip

\noindent $\bullet$ $C_n$:  Here 
 $\Pi=\{\epsilon_1-\epsilon_{2},\dots, \epsilon_{n-1}-\epsilon_{n},2\epsilon_n\}$,  $\Pi_0=\{\epsilon_{n-1}-\epsilon_n\}$ 
and thus
$$
\Delta_\mn^+=
\{\epsilon_i-\epsilon_n,\;\epsilon_i+\epsilon_n: 1\leq i\leq n-1\} 
\cup
\{\epsilon_i+\epsilon_j,\;1\leq i\le j\leq n-1\}. 
$$ 
This is a 2-step nilpotent Lie algebra and the sets of roots above correspond to 
linearly independent generators of $\mn/\rmc^1(\mn)$ and $\rmc^1(\mn) =\rmc_1(\mn)$ respectively. 
In this case, the inequality 
$\dim \rmc_1(\mn)+\dim\rmc^1(\mn)=n(n-1)\leq n(n-1)/2+n-1=\dim \mn$ holds only
if $n\leq 4$.
Therefore, it follows from Lemma \ref{lm1} that neither $\mn$ nor $\R\oplus\mn$
admit symplectic structures for $n\ge 5$.
Excluding odd-dimensional Lie algebras we have the next two cases.

If $n=3$, then $\mn$ is the quaternionic Heisenberg Lie algebra, 
it has dimension 7, with center of dimension 3 and 
$ \mathcal B=\{X_1,X_2,Y_1,Y_2\}\cup\{ Z_{11},Z_{12},Z_{22}\}$ is a basis with nonzero brackets 
$$[X_i,Y_j]=[X_j,Y_i]=Z_{ij},\quad 1\leq i\leq j\leq 2$$
and 
$$\omega=Z^{12}\wedge X^1+Z^{11}\wedge X^2+Z^{22}\wedge Y^2+U^1\wedge Y^1$$ 
defines a symplectic form in $\tilde\mn=\R U_1\oplus\mn$. 

If $n=4$,  then $\mn$ has dimension 12 with center of dimension 6. It admits a basis 
$\mathcal B=\{X_i,Y_i:\,1\le i \le 3\}\cup \{Z_{ij}:\,1\leq i\leq j\leq 3\}$ with nonzero brackets 
$$[X_i,Y_j]=[X_j,Y_i]=Z_{ij},\quad 1\leq i\leq j\leq 3$$
and
$$\omega=Z^{12}\wedge X^3+Z^{13}\wedge X^2+Z^{23}\wedge X^1+Z^{11}\wedge Y^1+Z^{22}\wedge Y^3+Z^{33}\wedge Y^2$$ 
defines a symplectic form in $\mn$.

\noindent $\bullet$ $G_2$: Here  
  $\Pi=\{\gamma_1,\gamma_2\}$, $\Pi_0=\{\gamma_1\}$ and 
$$\Delta_\mn=\{\gamma_1,\gamma_1+\gamma_2,2\gamma_1+\gamma_2,3\gamma_1+\gamma_2,3\gamma_1+2\gamma_2\}.$$
It tuns out that $\mn$ is the free 3-step nilpotent Lie algebra on 2 generators. 
It is known \cite{GB} that $\R \oplus\mn$ does not admit symplectic structures.

\medskip

It only remains to  consider the case $\Pi_0=\{\alpha,\beta\}$, which occurs in types $A$, $B$, $C$ or $D$.

\noindent $\bullet$ $A_n$:
In this case $\Pi_0=\{\gamma_k,\gamma_{k+1}\}$ with $k=1,\dots,n-1$ 
and it is straightforward to see that $\mn$ is 2-step nilpotent, $\dim\mn/\mn'=n$ and its center is of dimension $k(n-k)$.
It follows from Lemma \ref{lm1} that symplectic structures may appear only for 
$k=1,n-1$ (any $n\ge2$) or, $k=2$ if $n=4$.

If $k=1$ or $k=n-1$, then 
$\mn$ is odd dimensional and isomorphic to $\R X\ltimes_{\ad_X}(\R^{n-1}\oplus\R^{n-1})$ 
with  
\[
\ad_X=\begin{pmatrix}
0&0\\
I&0
\end{pmatrix},
\]
where $I$ denotes the $(n-1) \times (n-1)$ identity matrix. One easily proves that $\R\oplus \mn$ is symplectic for all $n$.
If $n=4$ and $k=2$, then $\mn$ 
is of dimension 8 with center of dimension 4. It admits a basis 
$\mathcal B=\{X_1,X_2,Y_1,Y_2\}\cup \{Z_{ij}:\,1\leq i,j\leq 2\}$ with nonzero brackets 
$$[X_i,Y_j]=Z_{ij},\quad 1\leq i, j\leq 2 $$
and
$$\omega=Z^{11}\wedge X^1+Z^{22}\wedge Y^2+Z^{12}\wedge X^1+Z^{21}\wedge X^2$$ 
defines a symplectic form in $\mn$.

\medskip

 \noindent $\bullet$ $B_2$: In this case $\mn$ is the nilradical of the Borel subalgebra of $\mgg$, it is 3-step nilpotent, has dimension 4
 and it is symplectic (cf. \cite{dB3}).
 
\medskip

 \noindent $\bullet$ $C_n$: Here 
 $\Pi=\{\epsilon_1-\epsilon_{2},\dots, \epsilon_{n-1}-\epsilon_{n},2\epsilon_n\}$, 
 $\Pi_0=\{\epsilon_{n-1}-\epsilon_{n},2\epsilon_n\}$ and
\begin{eqnarray*}
\Delta_\mn^+&=&
\{2\epsilon_n, \, \epsilon_i-\epsilon_n:\,1\leq i\leq n-1\} 
\,\cup\,
\{\epsilon_i+\epsilon_n:\,1\leq i\leq n-1\}\,
\cup\\
&&\hspace{5,5cm}\{\epsilon_i+\epsilon_j,\;1\leq i\le j\leq n-1\}. 
\end{eqnarray*}
This is a 3-step nilpotent Lie algebra, 
$\dim \rmc^1(\mn)=(n-1)+n(n-1)/2$, $\dim\rmc_1(\mn)=n(n-1)/2$ and $\dim \mn=n+(n-1)+n(n-1)/2$.
Therefore,  $\mn$ satisfies \eqref{eq.lm1} only for $n=3$. 

When $n=3$ the Lie algebra $\mn$ has dimension 8 (thus $\R\oplus\mn$ is not symplectic), $\dim\rmc^1(\mn)=5$
and we claim that it does not admit symplectic structures. 
Otherwise, since \eqref{eq.seriecent} holds for $t=2$, Corollary \ref{cor.seriecent} would
imply that there is a  highest weight vector in $\mgg_{(-2)}\wedge \mgg_{(-2)}$ 
(see \eqref{eq.g_{()}} and \eqref{eq.mm_j=mg_(-j)})  corresponding to a nonzero cohomology class in $H^2(\mn)$. 
However, it follows from Theorem  \ref{prop.H2_hwv} 
that the highest weight vectors corresponding to a nonzero cohomology class in $H^2(\mn)$
belong to $\mgg_{(-1)}\wedge \mn^-$. 
Therefore $\mn$ does not admit symplectic structures.

\medskip

\noindent $\bullet$ $D_n$: Here 
 $\Pi=\{\epsilon_1-\epsilon_{2},\dots, \epsilon_{n-1}-\epsilon_{n},\epsilon_{n-1}+\epsilon_n\}$,
 $\Pi_0=\{\epsilon_{n-1}-\epsilon_{n},\epsilon_{n-1}+\epsilon_n\}$ and 
 $\alpha= \epsilon_{n-1}+\epsilon_{n}$.
 In this case 
\begin{equation*}
\Delta_\mn^+=
\{ \epsilon_i-\epsilon_n,\, \epsilon_i+\epsilon_n:\,1\leq i\leq n-1\} 
\cup
\{\epsilon_i+\epsilon_j,\;1\leq i< j\leq n-1\}. 
\end{equation*} 
These are 2-step nilpotent Lie algebras with $\dim \rmc^1(\mn)=(n-1)(n-2)/2$ and $\dim \mn=2(n-1)+\dim\rmc^1(\mn)$. It follows from Lemma \ref{lm1} that neither $\mn$ nor $\R\oplus\mn$ 
admit symplectic structures unless $n\leq 5$.  
For $n=5$ one can prove that any closed 2-form is degenerate.
If $n=4$,  
then $\mn$ has dimension 9 with center of dimension 3. It admits a basis 
$\mathcal B=\{X_i,Y_i:\,1\le i \le 3\}\cup \{Z_{ij}:\,1\leq i< j\leq 3\}$ with nonzero brackets 
$$[X_i,Y_j]=-[X_j,Y_i]=Z_{ij},\quad 1\leq i< j\leq 3$$
and
$$\omega=Z^{12}\wedge X^3+Z^{13}\wedge X^2+Z^{23}\wedge X^1+Y^{1}\wedge Y^2+Y^{3}\wedge U^1$$ 
defines a symplectic form in $\tilde\mn=\R U_1\oplus \mn$.
\end{proof}

\end{document}